\documentclass[12pt]{article}
\usepackage{amsmath}
\usepackage{amssymb}
\usepackage{graphicx}

\usepackage{geometry}
\usepackage{enumitem}
\usepackage{tikz}
\usetikzlibrary{matrix,arrows}
\usetikzlibrary{calc}

\usepackage{todonotes}

\newcommand{\N}{\mathbb{N}}
\newcommand{\Z}{\mathbb{Z}}
\newcommand{\Q}{\mathbb{Q}}
\newcommand{\R}{\mathbb{R}}
\newcommand{\C}{\mathbb{C}}

\usepackage{mathtools}

\DeclareMathOperator{\id}{id}

\newcommand{\Loops}{\Omega}
\newcommand{\Realization}[1]{\abs{#1}}

\newcommand{\cat}[1]{\mathbf{#1}}
\DeclareMathOperator{\co}{co}
\newcommand{\w}{\mathrm{w}}

\newcommand{\wSdot}{\mathrm{w}\mathrm{S}_{\bullet}}
\newcommand{\wwSSdot}{\mathrm{w}\mathrm{w}\mathrm{S}_{\bullet}\mathrm{S}_{\bullet}}

\newcommand{\h}{\mathit{h}}

\newcommand{\hSdot}{\mathit{h}\mathrm{S}_{\bullet}}
\newcommand{\Sdot}{\mathrm{S}_{\bullet}}
\newcommand{\Sn}{\mathrm{S}_n}
\newcommand{\wSn}{\mathrm{w}\mathrm{S}_n}

\newcommand{\K}{K}
\newcommand{\KK}{\mathbb{K}}
\newcommand{\BK}{{\mathcal B}K}
\newcommand{\BKK}{{\mathcal B}\mathbb{K}}
\newcommand{\BKrel}{{\mathcal B}K^{\mathrm{rel}}}
\newcommand{\BKKrel}{{\mathcal B}\mathbb{K} ^{\mathrm{rel}}}
\newcommand{\BH}{{\mathcal B}H}
\newcommand{\PK}{{\mathcal P}K}
\newcommand{\PKrel}{{\mathcal P}K^{\mathrm{rel}}}
\newcommand{\PH}{{\mathcal P}H}
\newcommand{\Wh}{\mathbb{W}\mathrm{h}}
\newcommand{\BWh}{{\mathcal B}\mathbb{W}\mathrm{h}}
\newcommand{\BWhrel}{{\mathcal B}\mathbb{W}\mathrm{h}^{\mathrm{rel}}}

\newcommand{\B}{\mathcal B}

\newcommand{\BhSdot}{\mathcal{B}\mathit{h}\mathrm{S}_{\bullet}}

\DeclareMathOperator{\Ar}{Ar}
\DeclareMathOperator{\F}{F}
\newcommand{\defnword}[1]{\emph{#1}}
\newcommand{\nonneg}[1]{#1^{\geq 0}}
\newcommand{\norm}[1]{\left\|#1\right\|}
\newcommand{\abs}[1]{\left|#1\right|}

\DeclareMathOperator{\SIC}{\cal SIC}

\DeclareMathOperator{\FL}{FL}
\DeclareMathOperator{\FP}{FP}
\newcommand{\BFP}{{\mathcal B}\mathrm{FP}}

\DeclareMathOperator{\Ch}{\mathbf{Ch}}
\DeclareMathOperator{\BCh}{{\mathcal B}\mathbf{Ch}}
\DeclareMathOperator{\Fin}{fin}
\DeclareMathOperator{\hFin}{\textit{h}fin}
\DeclareMathOperator{\BhFin}{{\mathcal B}\textit{h}fin}
\DeclareMathOperator{\Bh}{{\mathcal B}\textit{h}}

\DeclareMathOperator{\Hom}{Hom}

\DeclareMathOperator{\ProjModules}{\mathbf{P}}
\DeclareMathOperator{\FreeModules}{\mathbf{F}}
\DeclareMathOperator{\ProjModulesWeighted}{\mathbf{P}^{\mathrm{w}}}
\DeclareMathOperator{\FreeModulesWeighted}{\mathbf{F}^{\mathrm{w}}}
\DeclareMathOperator{\BoundedProjModulesWeighted}{{\mathcal B}\mathbf{P}^{\mathrm{w}}}
\DeclareMathOperator{\BoundedFreeModulesWeighted}{{\mathcal B}\mathbf{F}^{\mathrm{w}}}
\DeclareMathOperator{\ProjModulesFingen}{\mathbf{P}_\mathrm{f}}
\DeclareMathOperator{\FreeModulesFingen}{\mathbf{F}_\mathrm{f}}

\DeclareMathOperator{\Cyl}{Cyl}
\DeclareMathOperator{\Cone}{Cone}

\DeclareMathOperator{\cofiber}{cofiber}

\DeclareMathOperator{\Asm}{Asm}
\DeclareMathOperator{\Monomial}{Monomial}
\newcommand{\semidirect}{\rtimes}
\newcommand{\Smash}{\wedge} 

\usepackage{hyperref}
\usepackage{aliascnt}
\usepackage[capitalise]{cleveref}

\usepackage{amsthm}
\newtheorem{theorem}{Theorem}
\newtheorem{conjecture}{Conjecture}

\newaliascnt{corollary}{theorem}

\aliascntresetthe{corollary}
\crefname{corollary}{Corollary}{Corollaries}

\newaliascnt{lemma}{theorem}
\newtheorem{lemma}[lemma]{Lemma}
\aliascntresetthe{lemma}
\crefname{lemma}{Lemma}{Lemmas}

\newaliascnt{proposition}{theorem}
\newtheorem{proposition}[proposition]{Proposition}
\aliascntresetthe{proposition}
\crefname{proposition}{Proposition}{Propositions}

\newaliascnt{claim}{theorem}
\newtheorem{claim}[claim]{Claim}
\aliascntresetthe{claim}
\crefname{claim}{Claim}{Claims}

\newtheorem*{approximationtheorem*}{Approximation Theorem}
\newtheorem*{splittingtheorem*}{Theorem \ref{theorem:splitting}}
\newtheorem*{theoremnontrivial*}{Theorem \ref{theorem:nontrivial}}
\newtheorem*{corollarysolvmanifold*}{Corollary \ref{corollary:solvmanifold}}

\theoremstyle{definition}
\newaliascnt{remark}{theorem}

\aliascntresetthe{remark}
\crefname{Remark}{Remark}{Remarks}

\newaliascnt{observation}{theorem}
\newtheorem{observation}[observation]{Observation}
\aliascntresetthe{observation}
\crefname{Observation}{Observation}{Observations}

\newaliascnt{example}{theorem}

\aliascntresetthe{example}
\crefname{Example}{Example}{Examples}

\newaliascnt{question}{theorem}

\aliascntresetthe{question}
\crefname{Question}{Question}{Questions}

\newaliascnt{definition}{theorem}
\newtheorem{definition}[definition]{Definition}
\aliascntresetthe{definition}
\crefname{Definition}{Definition}{Definitions}

\newaliascnt{axiom}{theorem}

\aliascntresetthe{axiom}
\crefname{Axiom}{Axiom}{Axioms}

\begin{document}
\title{{Bounded homotopy theory and the $K$-theory of weighted complexes}}
\date{}
\author{J.~Fowler, C.~Ogle\\
\textsc{Dept.~of Mathematics}\\
\textsc{The Ohio State University}
}
\maketitle

\begin{abstract}
  Given a bounding class $B$, we construct a bounded refinement
  $BK(-)$ of Quillen's $K$-theory functor from rings to spaces.
  $BK(-)$ is a functor from weighted rings to spaces, and is equipped
  with a comparison map $BK \to K$ induced by "forgetting control". In
  contrast to the situation with $B$-bounded cohomology, there is a
  functorial splitting $BK(-) \simeq K(-) \times BK^{rel}(-)$ where
  $BK^{rel}(-)$ is the homotopy fiber of the comparison map.
\end{abstract}

\tableofcontents

\section{Introduction}

Instead of considering all cocycles, one can restrict attention to
cocycles which are bounded with respect to a \textit{bounding class}
$\mathcal{B}$.  Forgetting that any condition was imposed on the
cocycles gives a natural map from $\mathcal B$-bounded cohomology to
ordinary group cohomology with coefficients---this yields a
\defnword{comparison map}
$$
\BH^\star(G;V) \to H^\star(G;V) 
$$
functorial in $G$ and $V$.  If this map is an isomorphism for all
\textit{suitable} coefficients modules $V$ (where suitable means
bornological modules over the rapid decay algebra $\mathcal
H_{\mathcal B,L}(G)$ as defined in \cite{2010arXiv1004.4677J}), we say
that $G$ is strongly $\mathcal B$-isocohomological
(abbr. $\B$-$\SIC$).  Properties of such comparison maps are related
to geometric properties of the group, e.g., surjectivity of the
comparison map is related to hyperbolicity when $\mathcal B = \{
\mbox{constant functions} \}$ \cite{MR1887670} and more general
isocohomologicality is related to combings \cite{2010arXiv1004.4677J}.

In light of the success of bounded methods in cohomology, the
precedent has been set to consider $\mathcal B$-bounded variants of
$K$-theory, and to introduce a $K$-theoretic comparison map
$\BK^\star(G) \to \K^\star(G)$.  We do so in \cref{section:k-theory}:
given a bounding class $\mathcal B$, we construct $\BK(-)$, a functor
from weighted rings\footnote{Although our theory applies generally to
  weighted rings, defined in \cref{section:weighted-rings}, our
  primary focus in this paper will be weighted rings of the form
  $R[G]$, where $R$ is a discretely normed ring and $G$ is a group
  with word length.} to spaces, and a comparison map
$$
\BK(-) \to \K(-)
$$
which is a natural transformation on the category of weighted rings.

Although similar in appearance to the ``forget control'' maps of
controlled $K$-theory (e.g., \cite{MR1311017}), the bounded $K$-theory
developed here is quite different, and in particular, is founded in what
might be called \textit{weighted algebraic topology.}  In contrast to the
comparison map in $\mathcal B$-bounded cohomology, we have
\begin{splittingtheorem*}
  There is a functorial splitting
  $$
  \BK(-) \simeq \K(-) \times \BKrel(-)
  $$
  where $\BKrel(-)$ is the homotopy fiber of the comparison map
  $\BK(-) \to \K(-)$.  The splitting extends to a splitting of spectra.
\end{splittingtheorem*}
Given the existence of the relative theory, it is a relevant question
as to whether or not it is nontrivial---in other words, are $\BK$-theory and
$\K$-theory actually different?  There is evidence to believe that the
following is true.
\begin{conjecture}
  Suppose a group $G$ is of type $\FL$ and is not $\mathcal B$-$\SIC$.
  Then $[C_\star(EG)] - \chi(G)$ represents a nonzero element
  in $\BKrel_0(\Z[G])$, where $EG$ is the homogeneous
  bar resolution of $G$.
\end{conjecture}
The relative group $\BKrel_0(\Z[G])$ represents a bounded version of
the Wall group, and measures precisely the obstruction of a
homotopically finite weighted chain complex over $\Z[G]$ to being
homotopically finite via a $\mathcal B$-bounded chain homotopy.  In
\cite{2010arXiv1004.4677J}, it was shown that there exists a closed
3-dimensional solvmanifold $M^3$ with $\pi = \pi_1 M$ for which there
exists an element $t^2 \in H^2(\pi;\C)$ not in the image of the
comparison map $\PH^2(\pi;\C) \to H^2(\pi;\C)$.  Thus, as a particular
case of the above conjecture, we formulate
\begin{conjecture}
  For $\pi = \pi_1 M^3$ as above, the class $$[C_\star(E\pi)] \neq 0
  \in \PKrel_0(\Z[\pi])$$ represents an element of infinite order,
  where $\mathcal P$ is the bounding class of polynomial functions.
\end{conjecture}
The theory presented here may be thought of as the ``linearized''
version of Waldhausen $K$-theory for weighted spaces, a topic we hope
to address more completely in some future work.  It is clear that much
more needs to be said about even the groups $\BK_0(\Z)$, which are at
this point completely unknown even for the polynomial bounding class.
This paper should be seen as an introduction to the theory.

\section{Bounded homotopies of weighted chain complexes}

\subsection{Weighted modules and bounding classes}

\subsubsection{Bounding classes}

We begin by recalling the definition of a bounding class
\cite{2010arXiv1004.4677J, MR2575390}.  Let $\mathcal{S}$ denote the
set of non-decreasing functions $\nonneg\R \to \nonneg\R$.  A
collection of functions $\mathcal{B} \subset \mathcal{S}$ is
\defnword{weakly closed} under the operation $\varphi : \mathcal{S}^n
\to \mathcal{S}$ if, for each $(f_1, \ldots,f_n) \in \mathcal{B}^n$,
there is an $f \in \mathcal{B}$ with $\varphi(f_1,\ldots,f_n) < f$.  A
\defnword{bounding class} is a subset $\mathcal{B} \subset
\mathcal{S}$ such that
\begin{enumerate}[label=($\mathcal{B}$C\arabic*),leftmargin=*]
\item $\mathcal{B}$ contains the constant function 1,
\item $\mathcal{B}$ is weakly closed under positive rational linear combinations, and
\item $\mathcal{B}$ is weakly closed under the operation $(f,g)
  \mapsto f \circ g$ for $f \in \mathcal{B}$ and $g \in \mathcal{L}$.
\end{enumerate}
Here, $\mathcal{L}$ denotes the linear bounding class $\left\{ f(x) =
  ax + b \middle\arrowvert a, b \in \nonneg\Q \right\}$.  Other
examples include the polynomial bounding class $\mathcal{P}$, the
bounding class $\mathcal{E}$ of simple exponential functions, and the
bounding class $\tilde{\mathcal{E}}$ of iterated exponential
functions.  It is easy to see that any bounding class $\mathcal B$ can
be closed under the operation of composition to form a bounding class
$\mathcal B'$ containing $\mathcal B$.

A bounding class $\mathcal{B}$ is \defnword{composable} if
$\mathcal{B}$ is weakly closed under the operation $(f,g) \mapsto f
\circ g$ for $f, g \in \mathcal{B}$.  The polynomial bounding class
$\mathcal{P}$ is composable; the exponential bounding class
$\mathcal{E}$ is not. Note, however, that any bounding class admits a closure
under the operation of composition, and thus for any $\mathcal B$ there is
(up to suitable equivalence) a smallest composable bounding class ${\mathcal B}'$
with ${\mathcal B}\subseteq {\mathcal B}'$.

We will write $\mathcal B' \preceq \mathcal B$ if, for every $f' \in
\mathcal B'$, there is an $f \in \mathcal B$ for which $f(x) \geq
f'(x)$ for all large $x$.  Also, we write $\mathcal B' \prec \mathcal
\mathcal B$ if $\mathcal B' \preceq \mathcal B$ and $\mathcal B
\not\preceq \mathcal B'$.

\subsubsection{Weights}

A \defnword{weighted set} $(X,w_X)$ is simply a set $X$ with a
function $w_X : X \to \nonneg\R$.  The weights are part of the data of
a weighted set, but whether a morphism of weighted sets $m : (X,w_X)
\to (Y,w_Y)$ is ``bounded'' depends on the choice of a bounding class
$\mathcal B$; a $\mathcal B$-bounded set map $m : (X,w_X) \to (Y,w_Y)$
is a map for which there exists $f \in \mathcal B$ so that
$$
w_Y(m(x)) \leq f(w_X(x))
$$
for all $x \in X$.

Note that when $X$ is finite, a morphism $m:(X,w_X)\to (Y,w_Y)$ is 
$\mathcal B$-bounded for any choice of bounding class $\mathcal B$ and
weight function $w_X$ on $X$. The distinction between
``bounded'' and ``unbounded'' only arises when the domain is
an infinite set.

Weighted sets can be considered in an equivariant context.  For a
group $G$ generated by $S$ where $S = S \cup S^{-1}$, there is a
natural notion of weight: a function $L : G \to \R^{\geq 0}$ is a
\defnword{length function} if $L(gh) < L(g) + L(h)$ and $L(g) =
L(g^{-1})$ for $g, h \in G$.  A length function is a \defnword{word
  length function} if $L(1) = 1$ and there is a function $\varphi : S
\to \nonneg\R$ so that
$$
L(g) = \min \left\{ \sum_{i=1}^n \varphi(x_i)
  \middle\arrowvert x_i \in S,\, x_1 x_2 \cdots x_n = g \right\}.
$$
Given a discrete group $G$ with length function $L$, a
\defnword{weighted $G$-set} is a weighted set $(S,w_S)$ with a
$G$-action on $S$, satisfying
$$
w_S(gs) \leq C \cdot \left( L(g) + w_S(s) \right)
$$
for all $g \in G$ and $s \in S$.  Analogous to the nonequivariant
case, when given a bounding class $\mathcal B$, we may consider the
$\mathcal B$-bounded maps of weighted $G$-sets.

\subsubsection{Free weighted modules}

We consider modules for which the elements are weighted; just as with
weighted sets, for each bounding class $\mathcal B$, we may consider
$\mathcal B$-bounded morphisms.
\begin{definition}
\label{definition:seminorms}
Let $R$ be a normed ring (in applications, $R$ will often be $\Z$),
Given a weighted set $(S,w_S)$, the free $R$-module $R[S]$ receives a
seminorm for every $f \in \mathcal B$, via
$$
\norm{\sum_{s \in S} \alpha_s s}_f =
  \sum_{s \in S} \abs{\alpha_s} f( w_S(s) ).
$$
With this setup, we call $R[S]$ a \defnword{weighted $R$-module}.  If
$(S,w_S)$ is weighted $G$-set, then $R[S]$ has the additional
structure of a \defnword{weighted $R[G]$-module}; again, for any
bounding class $\mathcal B$, the $R[G]$-module $R[S]$ can be equipped
with a collection of seminorms indexed by $\mathcal B$.
\end{definition}
One particular example will be important in applications: a
\defnword{free weighted $R[G]$-module} is a module of the form
$R[G][X] = R[G \times X]$, where $X$ is a weighted set $(X,w_X)$, and
$G \times X$ is the weighted $G$-set with weight
$$
w_{G \times X}(g,x) = L(g) + w_X(x).
$$

If $R[G][X_1]$ and $R[G][X_2]$ are two such free weighted $R[G]$-modules,
then their direct sum $R[G][X_1]\oplus R[G][X_2]$ is again a free weighted
$R[G]$-module via the identification 
\[
R[G][X_1]\oplus R[G][X_2]\cong R[G][X_1 \sqcup X_2]
\]
where the weight function on $X_1\sqcup X_2$ is the obvious one whose restriction
to $X_i$ is $w_{X_i}$.

Given two free weighted $R[G]$-modules, a natural next step is to consider
bounded maps between them---but bounded in what sense?  A map of free
$R[G]$-modules $\varphi : R[G][X] \to R[G][Y]$ is
\defnword{$\mathcal{B}$-bounded (in the sense of Dehn functions)} if
there exists $f \in \mathcal B$ so that for all $a \in R[G][X]$
$$
\norm{\varphi(a)}_{\id} \leq f\left( \norm{a}_{\id} \right)
$$
where $\norm{-}_{\id}$ means the weighted $\ell_1$-norm
$$
\norm{ \sum_{i} r_i s_i }_{\id} = \sum_{i} \abs{r_i} \cdot w_S(s_i).
$$

Alternatively, we say that $\varphi : R[G][X] \to R[G][Y]$ is
\defnword{$\mathcal{B}$-bounded (in the sense of functional analysis)}
if, for every $f \in \mathcal{B}$, there exists an $f' \in
\mathcal{B}$, so that for all $x \in R[G][X]$ the inequality $\norm{\varphi(x)}_f <
\norm{x}_{f'}$ holds.

This second notion (boundedness in the functional analytic sense) is in general
stronger than the first, but there are situations in which these two
notions agree.  For example, a $\mathcal B$-bounded map of sets $m :
(X,w_X) \to (Y,w_Y)$ induces a map $R[G][X] \to R[G][Y]$ which is
$\mathcal B$-bounded in either of the two senses.  Under mild
hypotheses on $R$ and on the bounding class $\mathcal B$, the same is
true for not necessarily based maps.
\begin{lemma}
  \label{boundedness-notions-agree}
  Consider two weighted sets $(X,w_X)$ and $(Y,w_Y)$, and suppose $R$
  is a normed ring, with the norm $\norm{-} : R \to (\epsilon,\infty)$
  where $\epsilon > 0$, and $\varphi : R[G][X] \to R[G][Y]$ is an
  $R[G]$-module map which is $\mathcal B$-bounded in the sense of Dehn
  functions.  If $\mathcal B \succeq \mathcal L$, then $\varphi$ is
  $\mathcal B$-bounded in the sense of functional analysis.
\end{lemma}
\begin{proof}
  By assumption, there exists $f \in \mathcal B$, so that
  $\norm{\varphi(a)}_{\id} \leq f\left( \norm{a}_{\id} \right)$ for
  all $a \in R[G][X]$.  One then verifies the following two claims.
\begin{description}
\item[Claim 1.] $\norm{\varphi(a)}_h \leq (h \circ f)( \norm{a}_{\id}
  )$.
\end{description}
Evaluating $\varphi$ on basis elements shows 
$\norm{\varphi(gx)}_h \leq (h \circ f)( \norm{gx}_{\id} ) =
\norm{gx}_{h \circ f}$.
\begin{description}
\item[Claim 2.] For a general element $a = \sum \lambda_{g,x} g x$ one
  has a sequence of inequalities 
\begin{align*}
\norm{\varphi(a)}_h
&\leq \sum_{g,x} |\lambda_{gx}| \cdot \norm{\varphi(gx)}_{h}  \\
&\leq \sum_{g,x} |\lambda_{gx}| \cdot \norm{gx}_{h \circ f} \\
&\leq \norm{a}_{h \circ f}.
\end{align*}
\end{description}
The arguments for these two claims is as given in Lemma~1 of
\cite{2010arXiv1004.4677J}.
\end{proof}
In light of \cref{boundedness-notions-agree}, we will assume that
$\mathcal B \succeq \mathcal L$ for the remainder of this paper.  When
we speak of $\mathcal{B}$-boundedness without any qualification, we
mean $\mathcal{B}$-bounded in the functional analytic sense; this is the
more natural notion from the bornological perspective.

For maps between not necessarily free weighted $R[G]$-modules, the
relationship between the two notions of boundedness the situation is less
clear, but we do have $\mathcal B$-boundedness for an important class
of morphisms.

\begin{proposition}
\label{prop:multiplication-is-bounded}
Let $S$ be a weighted $G$-set (on which the $G$-action is not
necessarily free), $R$ a normed ring, with the module $M = R[S]$
having a family of seminorms coming from a bounding class $\mathcal
B$, as in \cref{definition:seminorms}.  Then left multiplication by
any element of $R[G]$ is $\mathcal B$-bounded.

\end{proposition}
\begin{proof}[Proof (compare to Proposition~1 in \cite{2010arXiv1004.4677J}):]
Suppose $a = \sum_{g \in G} a_g\,g \in R[G]$ and $b = \sum_{s \in S}
b_s\,s \in R[S]$.  Given $f \in \mathcal B$,
choose $f_2 \in \mathcal B$ so that $f_2(x) \geq f(2x)$ and
choose $F \in \mathcal B$ so that $F(x) \geq \max \{ x, f_2(x)
\}$.

\begin{align*}
  \norm{ab}_f &= \norm{ \left( \sum_{g \in G} a_g\,g \right) \cdot
    \left( \sum_{s \in S} b_{s}\,s \right)}_f \\
&= \sum_{s \in S} \abs{ \sum_{g s' = s} a_g\,b_{s'} } f(w(s)) \\
&\leq \sum_{s \in S} \sum_{g s' = s} \abs { a_g\,b_{s'} } f\left(L(g) +
  w(s')\right) \\
&\leq \left( \sum_{s \in S} \sum_{\substack{g s' = s \\ L(g) \leq w(s')}}
\abs { a_g\,b_{s'} } f(2w(s')) \right) + \left( 
\sum_{s \in S} \sum_{\substack{g s' = s \\ L(g) \geq w(s')}}
\abs { a_g\,b_{s'} } f(2L(g)) \right) \\
&\leq \norm{ \sum_{g \in G} a_g g }_1 \norm{ \sum_{s \in s} b_s s }_{f_2} + 
\norm{ \sum_{g \in g} a_g g }_{f_2} \norm{ \sum_{s \in S} b_s s }_{1} \\
&\leq 2 \norm{ \sum a_g g}_F \norm{ \sum b_s s }_F = 2 \norm{ a }_F
\cdot \norm{b}_F.
\end{align*}
\end{proof}

\Cref{prop:multiplication-is-bounded} is true even when $X$ is an
infinite set; in the case of maps between finitely generated weighted
modules, much more is true.
\begin{proposition}
  \label{maps-are-bounded}
  Let $\mathcal{B}$ be a bounding class, $G$ a group with word length,
  and $X$ resp. $Y$ finite weighted sets.  Then every $R[G]$-module
  map $h : R[G][X] \to R[G][Y]$ is $\mathcal{B}$-bounded.
\end{proposition}
\begin{proof}
  The sets $X$ and $Y$ are finite; enumerate these sets, $X = \{ x_1,
  \ldots, x_n \}$ and $Y = \{ y_1, \ldots, y_m \}$.  We regard $h$
  as an $n$-by-$m$ matrix $(h_{ij})$ with entries in $R[G]$.

  Given $\alpha = \sum_{g \in G} \sum_{i=1}^n a_{g,x_i}\,g\,x_i \in
  R[G][X]$,
\begin{align*}
h(\alpha)
&= \sum_{g \in G} \sum_{i=1}^n a_{g,x_i}\, g\, h(x_i) 
= \sum_{g \in G} \sum_{i=1}^n \sum_{j=1}^m a_{g,x_i}\, g\, h_{ij} y_j.
\end{align*}
For $f \in \mathcal B$, choose $f_2 \in \mathcal B$ and $f_4 \in
\mathcal B$ so that $f_2(x) \geq f(2x)$ and $f_4(x) \geq f(4x)$.  Then,
\begin{align*}
\norm{h(\alpha)}_f
&= \norm{\sum_{g \in G} \sum_{i=1}^n \sum_{j=1}^m a_{g,x_i}\, g\,
  h_{ij} y_j}_f \\
&\leq \sum_{j=1}^m \norm{\sum_{g \in G} \sum_{i=1}^n a_{g,x_i}\, g\,
  h_{ij} y_j}_f \\
&\leq \sum_{j=1}^m 2 \cdot \norm{\sum_{g \in G} \sum_{i=1}^n a_{g,x_i}\, g\,
  h_{ij}}_{f_2} \norm{y_j}_{f_2} \\
&\leq C_f \cdot \sum_{j=1}^m \norm{\sum_{g \in G} \sum_{i=1}^n a_{g,x_i}\, g\,
  h_{ij}}_{f_2} \\
&\leq C_f \cdot \sum_{j=1}^m 2 \cdot \sum_{g \in G} \sum_{i=1}^n
\norm{a_{g,x_i}\, g}_{f_4} \, \norm{h_{ij}}_{f_4} \\
&\leq C_f \cdot 2 \cdot \sum_{g \in G} \sum_{i=1}^n
\norm{a_{g,x_i}\, g}_{f_4} \, H_{f_4} \\
&= 2 \, C_f \, H_{f_4} \, \sum_{g \in G} \sum_{i=1}^n
\norm{a_{g,x_i}\, g}_{f_4} \\
&\leq 2 \, C_f \, H_{f_4} \, C \, \sum_{g \in G} \sum_{i=1}^n
\norm{a_{g,x_i}\, g\, x_i}_{f_4} \\
&= 2 \, C_f \, H_{f_4} \, C \, \norm{\alpha}_{f_4}.
\end{align*}
Where $C_f = \max_{j}  \norm{y_j}_{f_2}$ and $H_{f_4} = \max_{i,j} \norm{h_{ij}}_{f_4}$ and $C = \max_{i}
1/w_X(x_i)$.
\end{proof}
More generally, \cref{maps-are-bounded} holds 
for infinite sets $X$ and $Y$ and any $R[G]$-module map
represented by a matrix with finitely many non-zero
entries.

\subsubsection{Projective weighted modules and admissible maps}
\label{section:proj-modules}

\begin{definition}
  A \defnword{weighted projective $R[G]$-module} is a pair $(M,p)$,
  where $M$ is a weighted free $R[G]$-module, and $p : M \to M$ is an
  $\mathcal L$-bounded projection map (meaning $p^2 = p$) which
  admits an $\mathcal L$-bounded section (recall that $\mathcal L$
  denotes the bounding class of non-decreasing linear functions).
\end{definition}
By \cref{maps-are-bounded}, if $M$ is finitely generated
as an $R[G]$-module, then any projection $p : M \to M$ is $\mathcal
L$-bounded and admits an $\mathcal L$-bounded section.

In general, given a free weighted $R[G]$-module $R[G][X]$ and a
submodule $M \subset R[G][X]$, the quotient $R[G][X]/M$ inherits an
obvious weighting---called the \defnword{induced weighting}---from the
weighting on $R[G][X]$, by defining the weight of an element to be the
infimum among representatives. This convention for weighting quotients of free
weighted modules extends to direct sums: if $M_i$ is a submodule of $R[G][X_i]$
for $i = 1,2$, then the direct sum 
\[
R[G_1]/M_1 \oplus R[G][X_2]/M_2
\]
inherits a weighting via identification with the quotient of
$R[G][X_1]\oplus R[G][X_2]$ by the submodule $M_1\oplus M_2$.

Given two weighted projective $R[G]$-modules, say $(M,p)$ and
$(N,q)$, a map $f : (M,p) \to (N,q)$ consists of a map $f : M
\to N$ which intertwines with the projections $p$ and $q$, i.e., a map
$f$ so that
\begin{center}
\begin{tikzpicture}
\node (m1) at (0,0) {\(M\)};
\node (m2) at (0,-2) {\(M\)};
\node (n1) at (2,0) {\(N\)};
\node (n2) at (2,-2) {\(N\)};
\draw[->] (m1) -- node[auto]{$p$} (m2);
\draw[->] (n1) -- node[auto]{$q$} (n2);
\draw[->] (m1) -- node[auto]{$f$} (n1);
\draw[->] (m2) -- node[auto]{$f$} (n2);
\end{tikzpicture}
\end{center}
commutes.

By \cref{maps-are-bounded}, any morphism between finitely generated
weighted projective $R[G]$-modules is bounded.  This need not be the
case for non-finitely-generated weighted projective $R[G]$-modules.

\begin{definition}
  \label{definition:admissible}
  An epimorphism $M \twoheadrightarrow N$ is \defnword{admissible} if
  it admits a linearly bounded section; a monomorphism $f : M
  \hookrightarrow N$ is admissible if the projection $N \to \left(
    \cofiber f \right)$ is admissible.  Here the cofiber of $f$ has
  the induced weighting.
\end{definition}
Admissibility guarantees that $\mathcal{B}\Hom_{\Z[G]}(-,V)$ sends a
short exact sequence with admissible maps to a short exact sequence.
With this restricted class of monomorphisms and epimorphisms, the
larger category of not necessarily finitely generated weighted modules
over the weighted ring $R[G]$ is an exact category.

\subsubsection{Categories of Modules}

The various modules we study can be packaged together into
categories.

\begin{definition}
  \label{defn:categories-of-modules}
  We summarize the categories we will be using.
  \begin{itemize}
  \item $\FreeModules(R[G])$ and $\ProjModules(R[G])$ denote the
    categories of free and projective $R[G]$-modules, respectively, with
    $R[G]$-module maps.
  \item $\FreeModulesFingen(R[G])$ and $\ProjModulesFingen(R[G])$
    denote the categories of 
    finitely generated free and finitely generated projective
    $R[G]$-modules, respectively, with $R[G]$-module maps.
  \item $\FreeModulesWeighted(R[G])$ and $\ProjModulesWeighted(R[G])$
    denote the categories of weighted free and weighted projective
    $R[G]$-modules, respectively, with (not necessarily bounded)
    $R[G]$-module maps.
  \item $\BoundedFreeModulesWeighted(R[G])$ and
    $\BoundedProjModulesWeighted(R[G])$ denote the categories of
    weighted free and weighted projective $R[G]$-modules,
    respectively, with $\mathcal B$-bounded $R[G]$-module maps.  For
    this to form a category, the morphisms need to be composable,
    which requires that the bounding class $\mathcal B$ be composable.
  \end{itemize}
\end{definition}

Unlike the first three cases, $\BoundedFreeModulesWeighted(R[G])$ and
$\BoundedProjModulesWeighted(R[G])$ are almost never abelian
categories, even if $\mathcal B$ is the bounding class $\mathcal
B_{\mathrm{max}}$ of all non-decreasing functions. But nevertheless,
in each of these categories, there is a notion of \defnword{zero
  morphism}, so one can construct chain complexes of objects in these
categories.

\subsubsection{For rings more generally}
\label{section:weighted-rings}

The above structures can be codified by the notion of a
\defnword{weighted ring}, meaning a ring $\tilde{R}$ equipped with two
norms: an $\ell_1$ norm and a weighted $\ell_1$ norm (corresponding to
the weighted $\ell_1$ norm coming from the word length function on
$G$).

A \defnword{weighted $\tilde{R}$-module $M$} is an $\tilde{R}$-module
similarly equipped with a pair of norms $\norm{-}_1$ and $\norm{-}_w$
satisfying
$$
\norm{r \cdot m }_1 \leq \norm{r}_1 \cdot \norm{m}_1
$$
and also (as in the proof of \cref{prop:multiplication-is-bounded}),
$$
\norm{r \cdot m }_w \leq \norm{r}_w \cdot \norm{m}_1 + \norm{r}_1 \cdot \norm{m}_w.
$$
A weighted ring $\tilde{R}$ is required to be an $\tilde{R}$-module,
with respect to both left and right multiplication.

So defined, the $K$-theoretic constructions introduced in the
following sections can be extended in a natural way to the more
general class of weighted rings.  However, for the purpose of this
paper, we will assume henceforth that our weighted rings $\tilde{R}$
are of the form $R[G]$ for a normed ring $R$ and a discrete group $G$
with word length function.

\subsection{Categories of complexes}
\label{section:categories-of-complexes}

Now we consider categories of chain complexes of weighted
$R[G]$-modules; let $\cat{C}$ denote one of the aforementioned
categories with zero morphisms (e.g., $\FreeModules(R[G])$,
$\ProjModules(R[G])$, $\FreeModulesFingen(R[G])$,
$\ProjModulesFingen(R[G])$, $\FreeModulesWeighted(R[G])$,
$\ProjModulesWeighted(R[G])$, $\BoundedFreeModulesWeighted(R[G])$, or
$\BoundedProjModulesWeighted(R[G])$).  The objects of the category
$\Ch(\cat{C})$ are the \defnword{chain complexes} of objects in the
category $\cat{C}$; the differentials in the chain complex are
morphisms in $\cat{C}$, and the morphisms between objects of
$\Ch(\cat{C})$ are the \defnword{chain maps}.

There are many variants of this construction: one may impose
finiteness conditions (e.g., one can demand that the chain complexes
be finite, or merely chain homotopy equivalent to a finite complex),
and, when the objects are weighted, one may demand that certain
aspects of the chain complexes be $\mathcal B$-bounded (e.g., that the
differentials, the chain maps, or the chain homotopies be bounded).
Notation for describing combinations of these conditions is
summarized in the following definition.
\begin{definition}
  \label{defn:categories-of-complexes}
  The following are subcategories of $\Ch(\cat{C})$.
  \begin{itemize}
  \item $\Ch_{\Fin}(\cat{C})$ denotes the full subcategory of chain
    complexes which are \defnword{finite}; a chain complex $(A_\star,d)$
    is finite if $\bigoplus_{n \in \Z} A_n$ is finitely generated over $R[G]$.
  \item $\Ch_{\hFin}(\cat{C})$ denotes the full subcategory of
    \defnword{homotopically finite} chain complexes; a chain complex
    is homotopically finite if it is chain homotopy equivalent to a
    finite chain complex.
  \end{itemize}
  If $\cat{C}$ is a category with weighted objects (e.g.,
  $\FreeModulesWeighted(R[G])$, $\ProjModulesWeighted(R[G])$,
  $\BoundedFreeModulesWeighted(R[G])$, or
  $\BoundedProjModulesWeighted(R[G])$) and $\mathcal B$ is a bounding
  class, then there are ``bounded'' subcategories of
  $\Ch(\cat{C})$ worth considering.
  \begin{itemize}
  \item The categories $\BCh(\cat{C})$, $\BCh_{\Fin}(\cat{C})$,
    and $\BCh_{\hFin}(\cat{C})$ have the same objects as
    $\Ch(\cat{C})$, $\Ch_{\Fin}(\cat{C})$, and
    $\Ch_{\hFin}(\cat{C})$, respectively, but the morphisms in
    categories prefixed by $\mathcal B$ are, degreewise, $\mathcal
    B$-bounded.
  \item $\BCh_{\BhFin}(\cat{C})$ is a full subcategory of
    $\BCh_{\hFin}(\cat{C})$; the objects of
    $\BCh_{\BhFin}(\cat{C})$ are chain homotopy equivalent to a
    finite complex via a $\mathcal B$-bounded chain homotopy.
  \end{itemize}
\end{definition}

\begin{observation}
  To understand how this notation is being used, one can consider the
  difference between $\BCh(\FreeModulesWeighted(R[G])$ and
  $\Ch(\BoundedFreeModulesWeighted(R[G])$. In the former category, the
  modules are not necessarily finitely generated, the chain complexes
  may have differentials which are not $\mathcal B$-bounded, but the
  chain maps are $\mathcal B$-bounded.  In the latter category, the
  differentials, being maps in $\BoundedFreeModulesWeighted(R[G])$ are
  $\mathcal B$-bounded, but the chain maps need not be $\mathcal
  B$-bounded.
\end{observation}

There are some obvious relationships between the above categories.
\begin{observation}
  \label{equivalences-between-categories}
  Applying \cref{maps-are-bounded}, the forgetful functor
  $$\BCh_{\Fin}(\BoundedProjModulesWeighted(R[G])) \to \Ch_{\Fin}(\ProjModulesWeighted(R[G]))$$
  is an isomorphism of categories, as every chain map in
  $\Ch_{\Fin}\ProjModulesWeighted(R[G]))$ is bounded (notice the
  subscript ``$\Fin$'' forces the chain complexes to be, degreewise,
  finitely generated).

  If $X$ is a finite set, any two different weight functions $w_1$ and
  $w_2$ on $X$ produce weighted $R[G]$-modules $R[G][(X,w_1)]$ and
  $R[G][(X,w_2)]$ which are, via the identity on $X$, canonically
  $\mathcal B$-boundedly isomorphic.  Consequently, the forgetful
  functor
  $$\Ch_{\Fin}(\ProjModulesWeighted(R[G])) \to \Ch_{\Fin}(\ProjModules(R[G]))$$
  is an equivalence of categories.

  Without the finiteness condition, $\Ch(\ProjModulesWeighted(R[G]))$ and
  $\Ch(\ProjModules(R[G]))$ are not equivalent categories.
\end{observation}
These above observations apply for free modules in place of projective
modules.

In the subcategories of $\BCh(\cat{C})$, two objects can be chain
homotopy equivalent in two different ways: there is the usual
(``coarse'') notion of chain equivalence, and the finer relation of
$\mathcal B$-bounded chain equivalence.  Waldhausen's setup of a
category with cofibrations and weak equivalences axiomatizes the
comparison of equivalences on a category; we review his setup now.

\section{Waldhausen $K$-theory of $\mathcal B$-bounded chain complexes}
\label{section:k-theory}

\subsection{Recap of Waldhausen $K$-theory}

\subsubsection{Waldhausen categories}

\newcommand{\ZeroObject}{\ast}
\newcommand{\cofibration}{\hookrightarrow}

A \defnword{Waldhausen category} (see \cite{MR802796}) involves two
distinguished classes of morphisms: cofibrations and weak
equivalences.  After the definition, we explain why these
distinguished classes exist in the categories discussed in
\cref{section:categories-of-complexes}.

\begin{definition}
  A \defnword{category with cofibrations} means a category $\cat{C}$,
  equipped with a zero object $\ZeroObject$ (both initial and
  terminal), together with a subcategory $\co\cat{C}$, the morphisms
  of which are called \defnword{cofibrations}, and are denoted by
  hooked arrows $\cofibration$.  The subcategory $\co\cat{C}$ is
  \defnword{wide}, meaning that the every object of $\cat{C}$ is an
  object of $\co\cat{C}$, but not every morphism is a cofibration.

The subcategory of cofibrations satisfies the following three
properties.
\begin{enumerate}[label=(Cof \arabic*),leftmargin=*]
\item\label{Cof1} Every isomorphism in $\cat{C}$ is in $\co\cat{C}$; in short,
  $\co\cat{C}$ is replete.
\item\label{Cof2} For every object $X$ in $\cat{C}$, the map $\ZeroObject \to X$ is in
$\co\cat{C}$.
\item\label{Cof3} \label{cobase-change}Cofibrations are preserved under co-base
  change, meaning that for any cofibration $i : X \cofibration Y$ and
  any morphism $f : X \to Z$ in $\cat{C}$, there is a pushout square
\begin{center}
\begin{tikzpicture}
\node (x) at (0,0) {\(X\)};
\node (y) at (0,-2) {\(Y\)};
\node (z) at (2,0) {\(Z\)};
\node (w) at (2,-2) {\(W\)};
\draw[right hook->] (x) -- node[auto]{$i$} (y);
\draw[->] (x) -- node[auto]{$f$} (z);
\draw[->] (y) -- (w);
\draw[right hook->] (z) -- node[auto]{$j$} (w);
\end{tikzpicture}
\end{center}
and the map $j : Z \cofibration W$ is a cofibration.
\end{enumerate} 
\end{definition}

Let $\cat{C}$ be one of the categories of modules listed above in
\cref{defn:categories-of-modules}.  In $\Ch(\cat{C})$, in the
unweighted setting, a cofibration is a degreewise monomorphism of
chain complexes which is degreewise split.

If $\cat{C}$ is a category with weighted objects and $\mathcal
B$-bounded maps, a chain map $f : C_\star \to D_\star$ of weighted
complexes in $\Ch(\cat{C})$ is a cofibration which is degreewise an
admissible monomorphism, meaning that there is a choice of cofiber
$E_\star = D_\star / C_\star$ so that for all $n \in \Z$ the map $D_n
\to E_n$ is an admissible epimorphism.  This yields a splitting
degreewise, but not necessarily a splitting on the level of chain
complexes.

\begin{lemma}
  \label{lemma:cofibration-structure}
  Let $\cat{C}$ be one of the categories of chain complexes listed in
  \cref{defn:categories-of-complexes}; using the preceding definition
  of a subcategory $\co\cat{C}$ of cofibrations, axioms \ref{Cof1},
  \ref{Cof2}, and \ref{Cof3} hold.
\end{lemma}
\begin{proof}
  For the classical case in which the chain maps are not weighted, the
  proof is standard.  When the objects are weighted and the morphisms
  are $\mathcal B$-bounded, properties \ref{Cof1} and \ref{Cof2} are
  again clear; the remaining issue is \ref{Cof3}.  The fact that it is
  a cofibration diagram comes one gets for free, moreover, if $f$ and
  $i$ are bounded, then $j$ bounded where $W$ has the induced
  weighting $(Z \oplus Y)/\sim$.  However, we need to know that $j$ is
  a cofibration, i.e., an admissible monomorphism.  Consider the
  diagram of weighted chain complexes
\begin{center}
\begin{tikzpicture}
\node (x) at (0,0) {\(X_\star\)};
\node (y) at (0,-2) {\(Y_\star\)};
\node (z) at (2,0) {\(Z_\star\)};
\node (w) at (2,-2) {\(W_\star\)};
\node (u1) at (0,-4) {\(U_\star\)};
\node (u2) at (2,-4) {\(U_\star\)};
\draw[right hook->] (x) -- node[auto]{$i$} (y);
\draw[->] (x) -- node[auto]{$f$} (z);
\draw[->] (y) -- (w);
\draw[right hook->] (z) -- node[auto]{$j$} (w);
\draw[->] (y) -- (u1);
\draw[->] (w) -- (u2);
\draw[double,double distance=2pt] (u1) -- (u2);
\draw[->] (u1) edge[out=180,in=180] node[left]{$i'$} (y);
\draw[->] (u2) edge[out=0,in=0] node[right]{$j'$} (w);
\end{tikzpicture}
\end{center}
Since $i$ is a cofibration, it is an admissible monomorphism, so there
is a degree-preserving section $i'$ of graded modules from its cofiber
$U$.  The top square is a pushout, so the cofiber of $i$ is the same
as the cofiber of $j$, and we get the required section $j'$ of graded
modules by following the diagram; thus, $j$ is an \textit{admissible}
monomorphism.
\end{proof}

\begin{definition}
  Given a category $\cat{C}$ with a subcategory $\co\cat{C}$ of
  cofibrations, a \defnword{category of weak equivalences} for
  $\cat{C}$ is a subcategory $\w\cat{C}$ which satisfies two
  properties.  
\begin{enumerate}[label=(Weq \arabic*),leftmargin=*]
\item \label{Weq1}Every isomorphism in $\cat{C}$ is in $\w\cat{C}$.
\item \label{Weq2}Weak equivalences can be glued together, meaning that if
\begin{center}
\begin{tikzpicture}
\node (b) at (0,0) {\(B\)};
\node (a) at (2,0) {\(A\)};
\node (c) at (4,0) {\(C\)};
\node (b') at (0,-2) {\(B'\)};
\node (a') at (2,-2) {\(A'\)};
\node (c') at (4,-2) {\(C'\)};
\draw[right hook->] (a) -- (b);
\draw[right hook->] (a') -- (b');
\draw[->] (a) -- (c);
\draw[->] (a') -- (c');
\draw[->] (b) -- node[auto]{$\simeq$} (b');
\draw[->] (a) -- node[auto]{$\simeq$} (a');
\draw[->] (c) -- node[auto]{$\simeq$} (c');
\end{tikzpicture}
\end{center}
where the arrows decorated with $\simeq$ are in $\w\cat{C}$, then the
induced map between pushouts $B \cup_A C \to B' \cup_A C'$ is also in
$\w\cat{C}$.
\end{enumerate}
\end{definition}
Again, the subcategory $\w\cat{C}$ is wide, meaning that every object
in $\cat{C}$ is in the subcategory of weak equivalences.

Let $\cat{C}$ be one of the categories of modules listed in
\cref{defn:categories-of-modules}.  Consider the subcategory
$\h\Ch(\cat{C})$ which has the same objects as $\Ch(\cat{C})$ but
whose morphisms are chain homotopy equivalences; doing so endows
$\Ch(\cat{C})$ with the structure of a category with weak
equivalences.  This is the classical case. In the presence of weighted
objects and a bounding class $\mathcal B$, there is a finer notion of
$\mathcal B$-bounded chain homotopy equivalence, denoted $\Bh$.  To
say a chain map $F : C_\star \to D_\star$ is a $\mathcal B$-bounded
chain homotopy equivalence means that there is a $\mathcal B$-bounded
chain homotopy inverse $G : D_\star \to C_\star$ so that $F \circ G$
and $G \circ F$ are $\mathcal B$-boundedly homotopic to the identity,
i.e., the chain homotopy is a $\mathcal B$-bounded map.

To summarize, there are three increasingly restrictive ways one can
introduce weak equivalences; given a category $\cat{C}$ of weighted
objects:
\begin{itemize}
\item $\h\Ch(\cat{C})$, in which weak equivalences are chain homotopy
  equivalences, and the weights are simply ignored;
\item $\h\BCh(\cat{C})$, in which weak equivalences are again possibly
  unbounded chain homotopy equivalences, but the chain maps are
  $\mathcal B$-bounded;
\item $\Bh\BCh(\cat{C})$, in which the weak equivalences are $\mathcal
  B$-bounded chain maps, for which the homotopies to the identity are
  also $\mathcal B$-bounded.
\end{itemize}

\begin{lemma}
  The axioms~\ref{Weq1} and~\ref{Weq2} are satisfied in the
  aforementioned categories.
\end{lemma}

\begin{proof}
  In the unweighted cases $\h\Ch(\cat{C})$ and $\h\BCh(\cat{C})$, this
  result is classical.  In the weighted case, \ref{Weq1} is satisfied
  because a $\mathcal B$-bounded isomorphism is, after forgetting the
  weights, an isomorphism.

  To verify axiom \ref{Weq2}, we note that the weighting as defined on
  $\left(B \oplus C/\sim\right)$ resp. $\left(B' \oplus
    C'/\sim\right)$ produce a suitably bounded map of pushouts
  $\left(B \oplus C/\sim\right) \to \left(B' \oplus C'/\sim\right)$.

  To see that this map is a $\mathcal B$-bounded chain homotopy
  equivalence, it suffices to verify the following technical fact:

  \begin{claim}
    \label{boundedly-acyclic}
    Given an admissible short exact sequence of weighted chain complexes
    $$A_\star \hookrightarrow B_\star \twoheadrightarrow C_\star$$
    with $A_\star \simeq \ZeroObject$ via a $\mathcal B$-bounded chain
    homotopy, the admissible epimorphism $B_\star
    \twoheadrightarrow C_\star$ is a $\mathcal B$-bounded chain
    homotopy equivalence.
  \end{claim}
This can be verified directly using exactly the same type of argument
as one uses in the unbounded case---the argument is left to the reader.
\end{proof}

\subsubsection{$K$-theory of a Waldhausen category}

We recall Waldhausen's $\Sdot$ construction.  The poset of integers
$[n] = \{ 0, 1, \ldots, n \}$ can be regarded as a category; the
category $\Ar[n]$ is the category of arrows in $[n]$.  Given a
category $\cat{C}$ with cofibrations $\co\cat{C}$, define $\Sn
\cat{C}$ to be the category of functors $A : \Ar[n] \to \cat{C}$, with
two properties.
\begin{enumerate}[label=(S\arabic*),leftmargin=*]
\item $A(j \to j) = \ZeroObject$
\item For a pair of composable arrows $i \to j$ and $j \to k$ in
  $\Ar[n]$, the map
$$
A(i \to j) \longrightarrow A(i \to k)
$$
is a cofibration, and the diagram
\begin{center}
\begin{tikzpicture}
\node (x) at (0,0) {\(A(i \to j)\)};
\node (y) at (0,-2) {\(A(j \to j) = \ZeroObject\)};
\node (z) at (4,0) {\(A(i \to k)\)};
\node (w) at (4,-2) {\(A(j \to k)\)};
\draw[right hook->] (x) -- node[auto]{} (y);
\draw[->] (x) -- node[auto]{$f$} (z);
\draw[->] (y) -- (w);
\draw[right hook->] (z) -- node[auto]{} (w);
\end{tikzpicture}
\end{center}
\end{enumerate}
The morphisms in the category $\Sn\cat{C}$ are the natural
transformations between such functors.  By collecting together (for
varying $n$) the categories $\Sn\cat{C}$, we form a simplicial
category $\Sdot\cat{C}$.

Canonically, $\Sn\cat{C}$ can be given the structure of a Waldhausen
category.  In particular, given a subcategory of weak equivalences
$\w\cat{C}$, the category $\Sn\cat{C}$ also has a subcategory of weak
equivalences $\wSn\cat{C}$; a natural transformation $A \to A'$ is a
weak equivalence if it is a weak equivalence objectwise. In this way,
one may form the basepointed simplicial space 
\[
\wSdot\cat(C) := \left\{ [n] \mapsto \Realization{\wSn\cat{C}} \right\}_{n\ge 0}
\]

The \defnword{Waldhausen $K$-theory space $\K(\cat{C})$} of $\cat{C}$
is defined to be $\Loops\Realization{\wSdot\cat{C}}$, which admits a
canonical delooping; we denote the associated spectrum by
$\KK(\cat{C})$.  The homotopy groups of
$\Loops\Realization{\wSdot\cat{C}}$ are the higher $K$-groups of the
Waldhausen category $\cat{C}$.

\subsubsection{Approximation theorem}

Among the tools developed by Waldhausen in \cite{MR802796} to study
his eponymous categories is his Approximation Theorem; stating this
powerful theorem, however, requires introducing some additional
properties that an arbitrary Waldhausen category may or may not
satisfy: these are the Saturation Axiom, and the Cylinder Axiom.

\newtheorem*{saturation-axiom}{Saturation Axiom}
\begin{saturation-axiom}
  If $f,g$ are composable maps in $\cat{C}$, and two of the three maps
  $f$, $g$, and $g \circ f$ are in $\w\cat{C}$, then the third is as well.
\end{saturation-axiom}

\begin{lemma}
  \label{saturation-lemma}
  The categories with weak equivalences,
  $\h\Ch_{\Fin}(\cat{C})$, $\h\Ch_{\hFin}(\cat{C})$, $\Bh\BCh_{\Fin}(\cat{C})$, $\Bh\BCh_{\BhFin}(\cat{C})$, satisfy the Saturation Axiom.
\end{lemma}
\begin{proof}
  If $f$ and $g$ are weak equivalences, then clearly so is $g \circ f$ in any
  of these categories.

  Suppose $f : A_\star \to B_\star$ and $g : B_\star \to
  C_\star$ are composable maps, and that $f$ and $g \circ f$ are weak
  equivalences.

  By $\widetilde{\Cone}(g \circ f)_\star$ and $\widetilde{\Cyl}(g
  \circ f)_\star$ we mean the pushouts
  \begin{center}
    \begin{tikzpicture}
      \node (00) at (0,0) {\(B_\star\)};
      \node (20) at (3,0) {\(\Cone(f)_\star\)};
      \node (02) at (0,-2) {\(\Cyl(g)_\star\)};
      \node (22) at (3,-2) {\(\widetilde{\Cone}(g \circ f)_\star\)};
      \draw[right hook->] (00) -- node[above]{} (20);
      \draw[right hook->] (00) -- node[above]{} (02);
      \draw[->>] (02) -- node[above]{} (22);
      \draw[->>] (20) -- node[above]{} (22);
    \end{tikzpicture} \hspace{1em}\raisebox{0.5in}{and}\hspace{1em}
    \begin{tikzpicture}
      \node (00) at (0,0) {\(B_\star\)};
      \node (20) at (3,0) {\(\Cyl(f)_\star\)};
      \node (02) at (0,-2) {\(\Cyl(g)_\star\)};
      \node (22) at (3,-2) {\(\widetilde{\Cyl}(g \circ f)_\star\)};
      \draw[right hook->] (00) -- node[above]{} (20);
      \draw[right hook->] (00) -- node[above]{} (02);
      \draw[->>] (02) -- node[above]{} (22);
      \draw[->>] (20) -- node[above]{} (22);
    \end{tikzpicture}
  \end{center}
  respectively.  One then has the following diagram
  \begin{center}
    \begin{tikzpicture}
      \node (a) at (0,0) {\(A_\star\)};
      \node (cylf) at (0,-2) {\(\Cyl(f)_\star\)};
      \node (conef) at (0,-4) {\(\Cone(f)_\star\)};
      \node (a') at (4,0) {\(A_\star\)};
      \node (fakecylgf) at (4,-2) {\(\widetilde{\Cyl}(g \circ f)_\star\)};
      \node (fakeconegf) at (4,-4) {\(\widetilde{\Cone}(g \circ f)_\star\)};
      \node (coneg) at (8,-4) {\(\Cone(g)_\star\)};
      \node (conezero) at (0,-5) {\(\ZeroObject\)};
      \node (fakeconezero) at (4,-5) {\(\ZeroObject\)};
\draw[right hook->] (a) -- (cylf);
\draw[right hook->] (a') -- (fakecylgf);
\draw[->>] (cylf) -- (conef);
\draw[->>] (fakecylgf) -- (fakeconegf);
\draw[double,double distance=2pt] (a) -- (a');
\path [-] (conef) edge[color=white] node[color=black,opaque,rotate=90] {$ \simeq
  $} (conezero);
\path [-] (fakeconegf) edge[color=white] node[color=black,opaque,rotate=90] {$ \simeq $} (fakeconezero);

\draw[right hook->] (cylf) -- (fakecylgf);
\draw[right hook->] (conef) -- (fakeconegf);
\draw[->>] (fakeconegf) -- (coneg);

    \end{tikzpicture}
  \end{center}
  where the bottom row represents an admissible short-exact sequence
  of complexes.  The fact that $f$ is a weak equivalence implies
  $\Cone(f)_\star \simeq \ZeroObject$ in either the bounded or
  unbounded setting, and similarly the fact that $g \circ f$ is a weak
  equivalence implies $\widetilde{\Cone}(g \circ f)_\star \simeq
  \ZeroObject$ in either the bounded or unbounded setting.  For
  unbounded complexes, this immediately implies that the cokernel of
  the bottom row is contractible.  In the $\mathcal B$-bounded case,
  we appeal to \cref{boundedly-acyclic} above to conclude that $\Cone
  (g)_\star$ is $\mathcal B$-boundedly contractible, implying that $g$
  is a weak equivalence.

  The final case, when $g$ and $g \circ f$ are weak equivalences follows in
  the same manner.
\end{proof}

Before stating the Approximation
Theorem, there are two more definitions that we need.
\begin{definition}
  Let $\cat{C}$ be a category with cofibrations and weak equivalences,
  and $\Ar \cat{C}$ the category of arrows in $\cat{C}$.  A
  \defnword{cylinder functor} on $\cat{C}$ is a functor from $\Ar
  \cat{C}$ to diagrams in $\cat{C}$, sending $f : A \to B$ to a
  diagram
\begin{center}
\begin{tikzpicture}
\node (a) at (0,0) {\(A\)};
\node (tf) at (2,0) {\(\Cyl(f)\)};
\node (b) at (4,0) {\(B\)};
\node (b') at (2,-2) {\(B\)};
\draw[->] (a) -- node[above]{$j_1$} (tf);
\draw[->] (b) -- node[above]{$j_2$} (tf);
\draw[->] (a) -- node[below]{$f$} (b');
\draw[->] (tf) -- node[auto]{$p$} (b');
\draw[->] (b) -- node[below,rotate=45]{$=$} (b');
\end{tikzpicture}
\end{center}
The object $\Cyl(f)$ is the \defnword{cylinder} of $f$ with $j_1$ and
$j_2$ corresponding to the \defnword{front inclusion} and
\defnword{back inclusion}, respectively, and $p$ corresponding the
natural \defnword{projection} to $B$.  The maps $j_1$ and $j_2$ are in
$\co\cat{C}$.  Moreover, the functor must satisfy
\begin{enumerate}[label=(Cyl \arabic*),leftmargin=*]
\item The front and back inclusions assemble to an exact functor $\Ar
  \cat{C} \to \F_1 \cat{C}$ sending $f : A \to B$ to $j_1 \vee j_2 : A
  \vee B \cofibration \Cyl(f)$.  The definition of $F_1 \cat{C}$ can be
  found in \cite{MR802796}.
\item $\Cyl(\ZeroObject \to A) = A$ for every object $A$ in $\cat{C}$; the
  two inclusions and projection map in the corresponding diagram are all
  the identity map on $A$.
\end{enumerate}
\end{definition}

\newtheorem*{cylinder-axiom}{Cylinder Axiom}
\begin{cylinder-axiom}
  For every $f : A \to B$ in $\cat{C}$, the projection $p : \Cyl(f) \to
  B$ is in $\w\cat{C}$.
\end{cylinder-axiom}

\begin{lemma}
  \label{cylinder-lemma}
  The categories $\Bh\Ch_{\Fin}$ and $\h\Ch_{\Fin}$ satisfy cylinder axiom.
\end{lemma}
\begin{proof}
  Given a chain map $f : C_\star \to D_\star$, define $\Cyl(f) :=
  \Cyl(f)_\star$ to be the algebraic mapping cylinder; in this case,
  the projection $p$ is a projection onto a summand, and therefore is
  bounded by virtue of the way that direct sums of weighted complexes
  are weighted.
\end{proof}

\begin{definition}
  Let $\cat{C}$ and $\cat{D}$ be categories with cofibrations and weak
  equivalences.  A functor $F : \cat{C} \to \cat{D}$ is an
  \defnword{exact functor} provided $F(\ZeroObject) = \ZeroObject$, $F$ sends
  weak equivalences to weak equivalences, cofibrations to
  cofibrations, and $F$ preserves the pushouts appearing in
  \ref{cobase-change}.
\end{definition}
There are many examples of exact functors.  For instance, the ``forget
control'' functor $\Ch_{\BhFin} \to \Ch_{\hFin}$ is exact.

The following is one of the fundamental results of Waldhausen
$K$-theory, and a key ingredient in the proof of
Theorem~\ref{theorem:splitting} below.

\begin{approximationtheorem*}[1.6.7 of \cite{MR802796}]
  Let $\cat{A}$ and $\cat{B}$ be categories with cofibrations and weak
  equivalences.  Suppose $\w\cat{A}$ and $\w\cat{B}$ satisfy the
  Saturation Axiom, that $\cat{A}$ has a cylinder functor, and that $\w\cat{A}$
  satisfies the Cylinder Axiom.  Let $F : \cat{A} \to \cat{B}$ be an
  exact functor with the approximation properties:
\begin{enumerate}[label=(App \arabic*),leftmargin=*]
\item\label{app1} $F$ reflects weak equivalences, meaning that a map
  is a weak equivalence in $\cat{A}$ iff its image is a weak
  equivalence in $\cat{B}$.
\item\label{app2} Given any object $A$ in $\cat{A}$ and a map $x : F(A) \to B$ in
  $\cat{B}$, there exists a cofibration $a : A \cofibration A'$ in
  $\cat{A}$ and a weak equivalence $x' : F(A') \to B$ in $\cat{B}$ for which
\begin{center}
\begin{tikzpicture}
\node (fa) at (0,0) {\(F(A)\)};
\node (fa') at (0,-2) {\(F(A')\)};
\node (b) at (2,-1) {\(B\)};
\draw[right hook->] (fa) -- node[left]{$F(a)$} (fa');
\draw[->] (fa) -- node[above]{$x$} (b);
\draw[->] (fa') -- node[below]{$x'$} (b);
\end{tikzpicture}
\end{center}
\end{enumerate}
commutes.

Then the induced maps $\Realization{\w\cat{A}} \to
\Realization{\w\cat{B}}$ and $\Realization{\wSdot\cat{A}} \to
\Realization{\wSdot\cat{B}}$ of pointed spaces are homotopy
equivalences, which extend to a map of spectra $\KK(\cat{A}) \to
\KK(\cat{B})$.
\end{approximationtheorem*}

\subsection{$K$-theory of the $\mathcal B$-bounded category of complexes}

Define the $K$-theory space of weighted complexes over $R[G]$ to be
$$
\BK(R[G]) =
\Loops\Realization{\BhSdot\BCh_{\hFin}(\BoundedProjModulesWeighted(R[G]))}
$$
with associated spectrum $\BKK(R[G])$.
This is the $\mathcal B$-bounded analogue to the algebraic $K$-theory
space for the ring $R[G]$:
$$
\K(R[G]) = \Loops\Realization{\hSdot
  \Ch_{\hFin}(\ProjModules(R[G]))}.
$$
There is an obvious natural transformation of infinite loop space
functors
$$\BK(-) \to \K(-)$$
induced by forgetting weights and bounds.  Finally, define the
relative $K$-theory $\BKrel(-)$ to be the homotopy fiber of $\BK(-)
\to \K(-)$.

\begin{theorem}
  \label{theorem:splitting}
  There is a functorial splitting of infinite loop spaces
  $$
  \BK(-) \simeq \K(-) \times \BKrel(-).
  $$
\end{theorem}
In other words, the $K$-theory of the category $\Ch_{\hFin}$ with
respect to the weak equivalence relation of $\mathcal{B}$-bounded
chain homotopy equivalence splits canonically as the product of the
$K$-theory of $\Ch_{\Fin}$ and the relative theory.
\begin{proof}
Compare the following to Proposition~2.1.1 of \cite{MR802796}.  To
conserve space, let $\cat{C} = \BoundedProjModulesWeighted(R[G])$.  We
begin by considering the diagram\footnote{This is not unlike the
  situation in equivariant homotopy theory, where one has various
  notions of weak equivalence.}:
\begin{center}
\begin{tikzpicture}
\node (02) at (0,-4)   {\(\Realization{\hSdot\Ch_{\Fin}(\cat{C})}\)};
\node (12) at (5,-4)   {\(\Realization{\hSdot\Ch_{\hFin}(\cat{C})}\)};

\node (11) at (5,-2)   {\(\Realization{\BhSdot\BCh_{\hFin}(\cat{C})}\)};

\node (00) at (0,0) {\(\Realization{\BhSdot\BCh_{\Fin}(\cat{C})}\)};
\node (10) at (5,0)  {\(\Realization{\BhSdot\BCh_{\BhFin}(\cat{C})}\)};

\draw[->] (00) -- node[auto]{$\simeq$} (10); 
\draw[->] (02) -- node[auto]{$\simeq$} (12);
\draw[->] (00) -- node[auto]{$\cong$} (02); 

\draw[->] (10) -- (11);
\draw[->] (11) -- (12);

\end{tikzpicture}
\end{center}

The left hand vertical map $\Realization{\BhSdot\BCh_{\Fin}(\cat{C})} \to
\Realization{\hSdot\Ch_{\Fin}(\cat{C})}$ is a weak homotopy equivalence; in fact,
more is true: the category $\Ch_{\Fin}(\cat{C})$ is the same as the category
$\BCh_{\Fin}(\cat{C})$ by \cref{maps-are-bounded}.  Furthermore, the two
choices of subcategories of weak equivalences, $\Bh\Ch_{\Fin}(\cat{C})$ and
$\h\Ch_{\Fin}(\cat{C})$, are identical, so $\Realization{\BhSdot\BCh_{\Fin}(\cat{C})} \to
\Realization{\hSdot\Ch_{\Fin}(\cat{C})}$ is a homeomorphism, induced by an isomorphism of
simplicial categories.

The top arrow $\Realization{\BhSdot\BCh_{\Fin}(\cat{C})} \to
\Realization{\BhSdot\BCh_{\BhFin}(\cat{C})}$ is a weak homotopy
equivalence by the Approximation Theorem; we verify properties
\ref{app1} and \ref{app2}.  Property \ref{app1} is clear:
$\BCh_{\Fin}(\cat{C})$ is a full subcategory of
$\BCh_{\BhFin}(\cat{C})$; moreover, if two objects in
$\BCh_{\Fin}(\cat{C})$ are $\mathcal B$-boundedly chain homotopy
equivalent in $\BCh_{\BhFin}(\cat{C})$, then they were so in
$\BCh_{\Fin}(\cat{C})$.  Similarly, the map $\Ch_{\Fin}(\cat{C}) \to
\Ch_{\hFin}(\cat{C})$ satisfies \ref{app1}.

The second property \ref{app2} is only slightly more involved.
Suppose $C_\star$ is a finite weighted complex, $D_\star$ a
$\mathcal B$-boundedly homotopically finite weighted complex, and $f :
C_\star \to D_\star$ a $\mathcal B$-bounded chain map.  Verifying
\ref{app2} requires factoring $f$ as
$$
C_\star \xhookrightarrow{g} E_\star \xrightarrow{h} D_\star
$$
with $g$ a cofibration, and $h$ a weak equivalence in $\Bh\BCh_{\BhFin}(\cat{C})$.

The chain complex $D_\star$ is $\mathcal B$-boundedly homotopically
finite; let $j : D_\star \to D'_\star$ be a $\mathcal B$-bounded
chain homotopy equivalence, with $D'_\star$ finite.  Define
$\tilde{f} = j \circ f$, and set $E_\star = \Cyl(\tilde{f})$.  Then
$E_\star$ is a finite complex, and the inclusion $g : C_\star
\hookrightarrow E_\star$ is a cofibration\footnote{The mapping
  cylinder construction together with the inclusion is the
  prototypical example of an admissible monomorphism in the category
  of $\mathcal B$-bounded chain complexes.} in $\BCh_{\Fin}(\cat{C})$.
The weak equivalence $h$ is the composition of the projection
$E_\star \to D'_\star$ (which is a weak equivalence) with $j^{-1}
: D'_\star \to D_\star$, which is $\mathcal B$-bounded because its
domain is finite (by \cref{maps-are-bounded}).

The same argument, albeit without considerations of $\mathcal
B$-boundedness, shows that $$\hSdot\Ch_{\Fin}(\cat{C}) \to
\BhSdot\Ch_{\hFin}(\cat{C})$$ which induces the bottom arrow after
realization, satisfies \ref{app2}.

To apply the Approximation Theorem, we also need the categories
involved to satisfy the Saturation Axiom: this we verified in
\cref{saturation-lemma}.  Finally, the hypotheses of the Approximation
Theorem require that $\BhSdot\BCh_{\Fin}(\cat{C})$ and
$\hSdot\Ch_{\Fin}(\cat{C})$ satisfy the Cylinder Axiom: this we
verified in \cref{cylinder-lemma}.

We therefore conclude by the Approximation Theorem that the top and
bottom horizontal maps are in fact weak equivalences, which in turn
implies that
$$
\Realization{\BhSdot\BCh_{\BhFin}(\cat{C})} \longrightarrow
\Realization{\BhSdot\BCh_{\hFin}(\cat{C})} \longrightarrow 
\Realization{\hSdot\Ch_{\hFin}}
$$
is a weak homotopy equivalence.  Hence
$\Realization{\hSdot\Ch_{\hFin}(\cat{C})}$ splits off
$\Realization{\BhSdot\BCh_{\hFin}(\cat{C})}$ up to homotopy.  These
maps are induced by maps of Waldhausen categories, and hence induce
infinite loop space maps upon passage to $K$-theory.
\end{proof}
On the level of spectra,
$$
  \BKK(-) \simeq \KK(-) \vee \BKKrel(-).
$$

\subsection{The relative Wall obstruction to $\mathcal B$-finiteness}

Inspired by Ranicki's setup for an algebraic finiteness obstruction
\cite{MR815431}, we now consider the relationship between whether
$C_\star(EG)$ vanishes in $\BKrel_0(R[G])$ and whether $C_\star(EG)$
has the $\mathcal B$-bounded homotopy type of a finite complex.

\begin{observation}
  Let $\mathcal B$ be a bounding class, and let $G$ be a group of type
  $\BFP^\infty$.  Then $C_\star(EG)$ is $\mathcal B$-boundedly
  homotopy equivalent to a finite complex $D_\star$, and so
  $[C_\star(EG)] = 0$ in $\BKrel_0(\Z[G])$.
\end{observation}

By Theorem~3 in \cite{2010arXiv1004.4677J}, a group is of type
$\BFP^\infty$ if and only if $G$ is $\mathcal B$-$\SIC$.
The contrapositive of the observation with this result proves
\begin{theorem}
  Let $\mathcal B$ be a bounding class, and $G$ a group of type
  $\FP^\infty$.  If $[C_\star(EG)] \neq 0$ in $\BKrel_0(\Z[G])$, then
  $G$ is not $\mathcal B$-$\SIC$.
\end{theorem}

A concrete example from \cite{2010arXiv1004.4677J} may also be
relevant here.  Specifically, there exists a solvable group $G$ (given
as a split extension $\Z^2 \to G \to \Z$) with $BG$ homotopy
equivalent to a closed oriented 3-manifold ${M_G}$.  Therefore, the
manifold $M_G$ is a finite model for $BG$, and via the spectral
sequence constructed in \cite{2010arXiv1004.4677J}, the group $G$ is
not $\mathcal B$-$\SIC$ for any bounding class $\mathcal{B} \prec
\mathcal{E}$.
\begin{conjecture}
  $[C_\star(EG)]$ represents an element of infinite order in
  $\PKrel_0(\Z[G])$.
\end{conjecture}

\section{An Assembly Map}

We construct an assembly map
$$
BG_{+} \Smash \BKK(\Z) \to \BKK(\Z[G]).
$$
by recognizing $\Loops^\infty \Sigma^\infty BG_{+}$ as the $K$-theory of
a Waldhausen category $\Monomial(G)$, and applying Section~1.5 of
\cite{MR802796} to promote a pairing of Waldhausen categories into a
product on the level of the associated spectra.

\subsection{Monomial category}

Recall that a monomial matrix is a square matrix which, when
conjugated by a permutation matrix, is diagonal.  Define $W_n(G)$ to
be the group of $n \times n$ monomial matrices with entries in $\pm
G$; or to be more precise, let $\Sigma_n$ denote the symmetric group
on $n$ letters.  These permutations act on $n \times n$ matrices, and
by interpreting $\pm G^n$ as the $n \times n$ diagonal matrices, the
group $\Sigma_n$ acts on $\pm G^n$ giving rise to the semidirect
product $W_n(G) = \Sigma_n \semidirect (\pm G^n)$.

The category $\Monomial(G)$ will package together these monomial
matrices $W_n(G)$ alongside projections and inclusions.  An object of
$\Monomial(G)$ is the $\Z[G]$-module $\Z[G][X]$ for some finite set $X$.
A morphism in $\Monomial(G)$ is an arbitrary composition of
\begin{itemize}
\item inclusions, $\Z[G][X] \to \Z[X \sqcup Y]$, induced from $X
  \hookrightarrow X \sqcup Y$
\item projections, $\Z[G][X \sqcup Y] \to \Z[G][X]$, sending $y \in Y$
  to zero, and
\item monomial maps, $\Z[G][X] \to \Z[G][X]$ given by an element of
  $W_n(G)$ when $n = |X|$.
\end{itemize}
Define a subcategory of cofibrations $\co\Monomial(G)$ by considering
maps given by arbitrary compositions of inclusions and monomial maps;
any such composition can be simplified to
$$
Z[G][X_1] \hookrightarrow Z[G][X_1 \sqcup X_2] \cong Z[G][X_1 \sqcup X_2]
$$
where the left hand map is an inclusion induced from $X_1
\hookrightarrow X_1 \sqcup X_2$ and the right hand isomorphism is a
monomial matrix in $W_n(G)$.  Define a subcategory of weak
equivalences $\w\Monomial(G)$ by considering only the monomial maps.
Then we have
\begin{lemma}
  The category $\Monomial(G)$ with the described subcategories of
  cofibrations and weak equivalences is a Waldhausen category.
\end{lemma}
\begin{proof}
  \ref{Cof1} and \ref{Cof2} are clear; considering the diagram
\begin{center}
\begin{tikzpicture}
\node (x) at (0,0) {\(\Z[G][X_1 \sqcup Y ]\)};
\node (y) at (0,-2) {\(\Z[G][X_1 \sqcup Y \sqcup X_2]\)};
\node (z) at (4,0) {\(\Z[G][X_1]\)};
\node (w) at (4,-2) {\(\Z[G][X_1 \sqcup X_2]\)};
\draw[right hook->] (x) -- node[auto]{$i$} (y);
\draw[->] (x) -- node[auto]{$f$} (z);
\draw[->] (y) -- (w);
\draw[right hook->] (z) -- node[auto]{$j$} (w);
\end{tikzpicture}
\end{center}  
verifies the co-base change axiom \ref{Cof3} when the top arrow is a
projection; an analogous argument verifies that \ref{Cof3} holds when
the top arrow is an inclusion or a monomial map.

It is immediate that every isomorphism is a weak equivalence, so
\ref{Weq1} holds.  That weak equivalences can be glued follows
directly by considering a few elementary cases; thus \ref{Weq2} holds.
\end{proof}
Because $\Monomial(G)$ is a Waldhausen category, we can apply the
$\Sdot$ construction to produce
$$
\K(\Monomial(G)) = \Loops\Realization{\wSdot \Monomial(G)}.
$$
But the identification of Waldhausen's $\Sdot$ construction with
Quillen's $Q$-construction and the Barratt--Priddy--Quillen--Segal
theorem yields
$$
\K(\Monomial(G)) \simeq \Loops B\left(\bigsqcup_{n \geq 0} BW_n(G)\right) \simeq \Z \times BW_\infty(G)^{+} \simeq \Loops^\infty \Sigma^\infty BG_{+}.
$$

\subsection{Pairing}

In Section~1.5 of \cite{MR802796}, Waldhausen describes how to build
external pairings of categories with cofibrations and weak equivalences.
\begin{proposition}
  \label{biexact-functor}
  Suppose the functor $F : \cat{A} \times \cat{B} \to \cat{C}$  is \defnword{bi-exact},
  meaning that $F(A, -)$ and $F(-,B)$ are exact functors for fixed
  objects $A$ of $\cat{A}$ and $B$ of $\cat{B}$, respectively.
  Additionally, suppose that for every pair of cofibrations $A
  \hookrightarrow A'$ in $\co\cat{A}$ and $B \hookrightarrow B'$ in
  $\co\cat{B}$, the induced map
  $$
  F(A',B) \cup_{F(A,B)} F(A,B') \to F(B,B')
  $$
  is a cofibration in $\cat{C}$.  Such a functor $F$ induces a map of
  bisimplicial bicategories
  $$
  \wSdot \cat{A} \times \wSdot \cat{B} \to \wwSSdot \cat{C}
  $$
  and further produces a map of spaces $K(\cat{A}) \Smash K(\cat{B})
  \to K(\cat{C})$ which extends to a map of spectra $\KK(\cat{A}) \Smash \KK(\cat{B})
  \to \KK(\cat{C})$
\end{proposition}
We now apply \cref{biexact-functor} to produce a pairing
$$
\KK(\Monomial(G)) \Smash \BKK(\Z) \to \BKK(\Z[G])
$$
by exhibiting a suitable biexact functor 
$$
F : \Monomial(G) \times \BCh_{\hFin}(\BoundedProjModulesWeighted(\Z))
\to \BCh_{\hFin}(\BoundedProjModulesWeighted(\Z[G])).
$$
Given $M \in \Monomial(G)$ and $C_\star \in
\BCh_{\hFin}(\BoundedProjModulesWeighted(\Z))$, define $F(M,C_\star)$
to be the chain complex $D_\star$ with
$$
D_n = M \otimes_{\Z} C_n.
$$
For a fixed $M \in \Monomial(G)$ or $C_\star \in
\BCh_{\hFin}(\BoundedProjModulesWeighted(\Z))$, the partial functors $F(M,-)$
and $F(-,C_\star)$ are exact, in other words, 
\begin{itemize}
\item $F(-,\ZeroObject) = F(\ZeroObject,-) = \ZeroObject$, 
\item $F(M,-)$ and $F(-,C_\star)$ send weak equivalences to weak equivalences,
\item $F(M,-)$ and $F(-,C_\star)$ send cofibrations to cofibrations, and 
\item $F(M,-)$ and $F(-,C_\star)$ preserves the pushouts appearing in
  \ref{cobase-change}.
\end{itemize}
There is also a technical condition to verify: given cofibrations $M
\hookrightarrow M'$ in $\co\Monomial(G)$ and $C_\star \hookrightarrow
C'_\star$ in $\co \BCh_{\hFin}(\BoundedProjModulesWeighted(\Z))$, is
the map
$$
F(M',C_\star) \cup_{F(M,C_\star)} F(M,C'_\star) \to F(M',C'_\star)
$$
a cofibration?  In fact it is.  
The cofibrations give rise to splittings $C'_p \cong C_p \oplus C''_p$ and
$M' \cong M \oplus M''$, so the degree $p$ component of $F(M',C_\star) \cup_{F(M,C_\star)}
F(M,C'_\star)$ is
$$
(M' \otimes C_\star)_p \oplus_{(M \otimes C_\star)_p} (M \otimes C'_\star)_p.
$$
or equivalently
$$
\left( \left( M \oplus M'' \right) \otimes C_p \right)
\oplus_{M \otimes C_p} \left(M \otimes \left(C_p \oplus C''_p\right) \right)
$$
which expands to
$$
\left(M \otimes C_p\right) \oplus 
\left(M'' \otimes C_p\right) \oplus 
\left(M \otimes C''_p\right)
$$
and so the map into 
\begin{align*}
F(M',C'_\star)_p &= (M' \otimes C'_\star)_p   \\
&= \left(M \oplus M''\right) \otimes \left(C_p \oplus C''_p\right)
\end{align*}
is a cofibration, as required by the hypotheses of
\cref{biexact-functor}.  Therefore, we have proved
\begin{proposition}
  The functor $F$ induces a pairing on the level of spectra
$$
\KK(\Monomial(G)) \Smash \BKK(\Z) \to \BKK(\Z[G]).
$$
which we denote by $\Asm(G)$.
\end{proposition}

\subsection{Whitehead spectrum}

The assembly map
$$
\Asm(G) : \Sigma^\infty BG_{+} \Smash \BKK(\Z) \to \BKK(\Z[G])
$$
permits us to define a $\mathcal B$-bounded Whitehead spectrum,
$$
\BWh(G) = \cofiber \Asm(G). 
$$
Consider the following diagram.
\begin{center}
\begin{tikzpicture}
\node (m1) at (0,0) {\(\Sigma^\infty BG_{+} \Smash \BKKrel(\Z)\)};
\node (m2) at (4,0) {\(\BKKrel(\Z[G])\)};
\node (m3) at (8,0) {\(\BWhrel(G)\)};

\node (n1) at (0,-1.5) {\(\Sigma^\infty BG_{+} \Smash \BKK(\Z)\)};
\node (n2) at (4,-1.5) {\(\BKK(\Z[G])\)};
\node (n3) at (8,-1.5) {\(\BWh(G)\)};

\node (p1) at (0,-3) {\(\Sigma^\infty BG_{+} \Smash \KK(\Z)\)};
\node (p2) at (4,-3) {\(\KK(\Z[G])\)};
\node (p3) at (8,-3) {\(\Wh(G)\)};
\draw[->] (m1) -- (m2);
\draw[->] (m2) -- (m3);
\draw[->] (n1) -- (n2);
\draw[->] (n2) -- (n3);
\draw[->] (p1) -- (p2);
\draw[->] (p2) -- (p3);

\draw[->] (m1) -- (n1);
\draw[->] (n1) -- (p1);

\draw[->] (m2) -- (n2);
\draw[->] (n2) -- (p2);

\draw[->] (m3) -- (n3);
\draw[->] (n3) -- (p3);

\end{tikzpicture}
\end{center}
By the functoriality of the splitting $\BKK(-) \simeq \KK(-) \vee
\BKKrel(-)$, the fiber of the vertical arrow $\BWh(G) \to \Wh(G)$ can
be identified with the cofiber $\BWhrel(G)$ of 
$$
\Asm(G) : \Sigma^\infty BG_{+} \Smash \BKKrel(\Z) \to \BKKrel(\Z[G])
$$
and further
\begin{theorem}
  There is a functorial splitting
  $$
  \BWh(-) \simeq \Wh(-) \times \BWhrel(-).
  $$
\end{theorem}

\subsection{Concluding Remarks}

The assembly map focuses attention on $\BK_\star(\Z)$, which we
conjecture to be highly nontrivial, even in degree zero.  To
illustrate some of the complexities involved, consider the map of
weighted sets
$$
f: (\N,\id) \to (\N,\log)
$$
where $(\N,\id)$ is the weighted set in which $n$ has weight $n$,
$(\N,\log)$ is the weighted set in which $n$ has weight $\log n$, and
$f(n) = n$.  The map $f$ is polynomially bounded, but the inverse
is not.  This map $f$ gives rise to a map of weighted $\Z$-modules
$$
\Z[\N] \to \Z[\N]
$$
which is unboundedly an isomorphism, but not invertible as a
polynomially bounded map.  For the polynomial bounding class
$\mathcal{P}$, the group $\PK_0(\Z)$ includes classes arising from
finitely generated free $\Z$-modules but also a class for the chain
complex
$$
0 \to \Z[\N] \to \Z[\N] \to 0
$$
We conjecture that the class of this chain complex is a nonzero
element of infinite order in $\PK_0(\Z)$.

\newpage

\bibliographystyle{alpha}
\bibliography{references}

\end{document}